\documentclass{amsart}
\usepackage{graphicx}
\usepackage{graphicx}
\usepackage{amssymb}
\usepackage{amsfonts}
\usepackage{amsmath}
\usepackage{amsthm}
\newtheorem{thm}{Theorem}[section]

\newtheorem{lem}[thm]{Lemma}
\newtheorem{rem}[thm]{Remark}

\numberwithin{equation}{section}

\newcommand{\cF}{\mathcal{F}}

\begin{document}
\title{Numerical schemes for $G$--Expectations}
 \thanks{I am very grateful to M.Nutz, H.M.Soner and O.Zeitouni for valuable discussions.}

 \author{Yan Dolinsky\\
 Department of Mathematics\\
 ETH, Zurich\\
 Switzerland }%

\address{
 Department of Mathematics, ETH Zurich 8092, Switzerland\\
 {e.mail: yan.dolinsky@math.ethz.ch}}

\date{\today}
\begin{abstract}
We consider a discrete time analog of $G$--expectations and
we prove that in the case where the time step goes to $0$
the corresponding values converge to the original
$G$--expectation. Furthermore we provide error estimates for
the convergence rate. This paper is continuation of \cite{DNS}. Our main tool is a strong approximation
theorem which we derive for general discrete time martingales.
\end{abstract}


\maketitle
\markboth{Y.Dolinsky}{Approximations of $G$--Expectations}
\renewcommand{\theequation}{\arabic{section}.\arabic{equation}}
\pagenumbering{arabic}

\section{Introduction}\label{sec:1}\setcounter{equation}{0}
In this paper we study numerical schemes for $G$--expectations,
which were introduced recently by Peng (see \cite{P1} and \cite{P2}).
A $G$--expectation is a sublinear function which maps random
variables on the canonical space $\Omega:=C([0,T];\mathbb{R}^d)$ to the real numbers.
The motivation to study $G$--expectations comes from mathematical finance,
in particular from risk measures (see \cite{NS} and \cite{P3}) and pricing
under volatility uncertainty (see \cite{DM} ,\cite{NS} and \cite{STZ}).

Our starting point is the dual view on $G$--expectation via volatility uncertainty
(see \cite{DHP}), which yields the representation $\xi\rightarrow \sup_{P\in\mathcal{P}} E_P[\xi]$
where $\mathcal{P}$ is the set of probabilities on $C([0,T];\mathbb{R}^d)$
such that under any $P\in\mathcal{P}$, the canonical process $B$ is a martingale with volatility
$d\langle B\rangle/dt$ taking values in a compact convex subset $\textbf{D}\subset\mathbb{S}^d_{+}$
of positive definite matrices. Thus the set $\textbf{D}$ can be understood
as the domain of (Knightian) volatility uncertainty and the
functional above represents the European option (with reward $\xi$) super--hedging price. For details
see (\cite{DM} and \cite{NS}).

In the current work we assume that $\xi$ is of the form
$F(B,\langle B\rangle)$ where $F$ is a path--dependent functional which satisfies some regularity
conditions.
In particular, $\xi$ can represent an award of path dependent European contingent claim.
In this case the reward
is a functional of the stock price which is equal to the
Doolean exponential of the canonical process,
and so quadratic variation appears naturally.

In \cite{DNS}, the authors introduced a volatility uncertainty in
discrete time and an analog of the Peng $G$--expectation. They
proved that the discrete time values converge to the continuous time
$G$--expectation. The main tools that were used there are the weak
convergence machinery together with a randomization technique. The
main disadvantage of the weak convergence approach is that it can
not provide error estimates. In order to obtain error estimates we
should consider all the market models on the same probability space,
and so methods based on strong approximation theorems come into
picture. In this paper we consider a bit different (than in
\cite{DNS}) discrete time analog of $G$--expectation and prove that
in a case where the time step goes to $0$ the corresponding values
converge to the original $G$--expectation. Furthermore, by deriving
a strong invariance principle for general discrete time martingales,
we are able to provide error estimates for the convergence rate.

The paper is organized as following.
In the next section we introduce the setup and formulate the
main results. In Section 3 we present the main machinery which we are use, namely we obtain
a strong approximation theorem for general martingales.
In Section 4 we derive auxiliary lemmas that we use
for the proof of the main results. In Section 5 we complete
the proof of Theorems \ref{thm2.1}--\ref{thm2.2}.

\section{Preliminaries and main results}\label{sec:2}\setcounter{equation}{0}
We fix the dimension $d\in\mathbb{N}$ and denote by $||\cdot||$ the
sup Euclidean norm on $\mathbb{R}^d$. Moreover, we denote by
$\mathbb{S}^d$ the space of $d\times d$ symmetric matrices and by
$\mathbb{S}^d_{+}$ its subset of nonnegative definite matrices.
Consider the space $\mathbb{S}^d$ with the operator norm
$||A||=\sup_{||v||=1}||A(v)||$. We fix a nonempty, convex and
compact set $\textbf{D}\subset\mathbb{S}^d_{+}$; the elements of
$\textbf{D}$ will be the possible values of our volatility process.
Denote by $\Omega={C}([0,T];\mathbb{R}^d)$ and
$\Gamma={C}([0,T];\mathbb{S}^d)$, the spaces of continuous functions
with values in $\mathbb{R}^d$ and $\mathbb{S}^d$, respectively. We
consider these spaces with the sup norm $||x||=\sup_{0\leq t\leq
T}||x_t||.$ Let $F:\Omega\times\Gamma\rightarrow\mathbb{R}$ be a
function which satisfies the following assumption. There exits
constants $H_1,H_2>0$ such that
\begin{eqnarray}\label{2.1}
&|F(u_1,v_1)-F(u_2,v_2)|\leq H_1 \exp\left(H_2(||u_1||+||u_2||+||v_1||+||v_2||)\right)\times\\
&(||u_1-u_2||+||v_1-v_2||), \ \ u_1,u_2\in\Omega, \ \ v_1,v_2\in\Gamma.\nonumber
\end{eqnarray}
Without loss of generality
we assume that the maturity date $T=1$.
We denote by $B=(B_t)_{0\leq t\leq 1}$ the
canonical process (on the space $\Omega$) $B_t(\omega)=\omega_t$, $\omega\in\Omega$ and by
$\cF_t:=\sigma(B_s,\, 0\leq s\leq t)$ the canonical filtration.
A probability measure $P$ on $\Omega$ is called a \emph{martingale law}
if $B$ is a $P$-martingale (with respect to the filtration $\mathcal{F}_t$) and $B_0=0$ $P$-a.s. (all our martingales start at the origin).
We set
\begin{equation}\label{2.2}
\mathcal{P}_{\textbf{D}}=\{P \ \mbox{martingale} \ \mbox{law} \ \mbox{on} \ \Omega: \ d\langle B\rangle/ dt\in \textbf{D}, \ P\times dt \ \mbox{a.s.}\},
\end{equation}
observe that under any measure ${P}\in \mathcal{P}_D$
the stochastic processes $B$ and $\langle B \rangle$, are random elements in $\Omega$
and $\Gamma$, respectively.
Consider the $G$--expectation
\begin{equation}\label{2.4}
V=\sup_{{P}\in\mathcal{P}_{\textbf{D}}}{{E}_P}F(B,\langle B \rangle)
\end{equation}
where ${E}_P$ denotes the expectation with respect to $P$.
A measure $P\in \mathcal{P}_{\textbf{D}}$ will be called $\epsilon$--optimal if
\begin{equation}\label{2.5}
V<\epsilon+{E}_PF(B,\langle B \rangle).
\end{equation}
Our goal is to find discrete time approximations for $V$. The advantage of discrete time approximations
is that the corresponding values can be calculated by dynamical programming. Furthermore, we will apply these approximations
in order to find $\epsilon$--optimal measures in the continuous time setting.

\begin{rem}
Let
$S={\{(S^1_t,...,S^d_t)\}}_{t=0}^1$ be
the Doolean's exponential $\mathcal{E}(B)$ of the canonical process $B$, namely
$S^i_t=S^i_0\exp\left(B^i_t-\langle B^i\rangle_t\right)$, $i\leq d$, $t\in [0,1]$.
The stochastic process $S$
represents the stock prices in a financial model
with volatility uncertainty. Clearly any random variable
of the from $g(S)$ where $g:C([0,T];\mathbb{R}^d)\rightarrow\mathbb{R}_{+}$
is a Lipschitz continuous function, can be written in the form $g(S)=F(B,\langle B\rangle)$
for a suitable $F$ which satisfies (\ref{2.1}). Thus we see that our setup includes payoffs
which correspond to path dependent European options.
\end{rem}

Next, we formulate the main approximation results. Let $\nu$ be a distribution
on $\mathbb{R}^d$ which satisfies the following
\begin{equation}\label{2.6}
\int_{\mathbb{R}^d}xd\nu(x)=0 \ \ \mbox{and} \ \
\int_{\mathbb{R}^d}x^i x^j d\nu(x)=\delta_{ij}, \ \ 1\leq i,j\leq d
\end{equation}
where $\delta_{ij}$ is the Kronecker--Delta. Furthermore, we assume that
the moment generating function $\psi_{\nu}(y):=\int_{x\in\mathbb{R}^d}\exp(\sum_{i=1}^d x^iy^i)d\nu(x)<\infty$
exists for any $y\in\mathbb{R}^d$ and, for any compact set $K\subset\mathbb{R}^d$
\begin{equation}\label{2.6+}
\sup_{n\in\mathbb{N}}\sup_{y\in K}\psi^n_{\nu}\left(\frac{y}{\sqrt n}\right)<\infty.
\end{equation}
Observe that the standard $d$--dimensional normal distribution $\nu=N(0,I)$
is satisfying the assumptions (\ref{2.6})--(\ref{2.6+}).

Let $n\in\mathbb{N}$ and  $Y_1,...,Y_n$ be a sequence of i.i.d. random vectors
with $\mathcal{L}(Y_1)=\nu$, i.e., the distribution
of the random vectors is $\nu$. We denote by $\mathcal{A}^{\nu}_n$ the set of all $d$--dimensional
stochastic process $M=(M_0,...,M_n)$ of the form, $M_0=0$ and
\begin{equation}\label{2.7}
M_i=\sum_{j=1}^i\frac{1}{\sqrt n}\phi_j(Y_1,...,Y_{j-1})Y_j, \ \ 1\leq i\leq n
\end{equation}
where $\phi_j:(\mathbb{R}^{d})^{j-1}
\rightarrow\sqrt{\textbf{D}}:=
\{\sqrt A: A\in\textbf{D}\}$ and $Y_1,...,Y_n$ are column vectors. As usual for a matrix
$A\in\mathbb{S}^d_{+}$ we denote by $\sqrt{A}$ the unique square root in $\mathbb{S}^d_{+}$.
Observe that $M$ is a martingale under the filtration which is generated by $Y_1,...,Y_n$.
Let $\langle M \rangle $ be the ($\mathbb{S}^d_{+}$ valued)
predictable variation of $M$. In view of (\ref{2.6}) we get
\begin{equation}\label{2.8}
\langle M\rangle_k=\frac{1}{n}\sum_{j=1}^k \phi^2_j(Y_1,...,Y_{j-1}), \ \ 1\leq k \leq n
\end{equation}
and we set $\langle M\rangle_0=0$.
Let $\mathcal{W}_n:(\mathbb{R}^d)^{n+1}\times(\mathbb{S}^d)^{n+1}\rightarrow\Omega\times\Gamma$
be the linear interpolation operator given by
$$
\mathcal{W}_n(u,v)(t):=
\left(\left[{nt}\right]+1-{nt}\right)(u_{\left[nt\right]},v_{\left[nt\right]})+
\left(nt-\left[nt\right]\right)(u_{\left[nt\right]+1},v_{\left[nt\right]+1}),
 \ {t}\in[0,1]
$$
where $u=(u_0,u_1,...,u_n)$, $v=(v_0,v_1,...,v_n)$ and
$[z]$ denotes the integer part of $z$.
Set
\begin{equation}\label{2.9}
V^{\nu}_n=\sup_{M\in\mathcal{A}^{\nu}_n}\mathbb{E}F\left(\mathcal{W}_n(M,\langle M\rangle)\right),
\end{equation}
we denote by $\mathbb{E}$ the expectation with respect to the underlying probability measure.
The following theorem which will be proved in Section \ref{sec5} is the main result of the paper.
\begin{thm}\label{thm2.1}
For any $\epsilon>0$ there exists a constant $\mathcal{C}_{\epsilon}=\mathcal{C}_{\epsilon}(\nu)$
which depends only on the distribution $\nu$ such that
\begin{equation}\label{2.10}
|V^{\nu}_n-V|\leq \mathcal{C}_{\epsilon}n^{\epsilon-1/8}, \ \ \forall{n}\in\mathbb{N}.
\end{equation}
Furthermore, if the function $F$ is bounded, then there exists a constant $\mathcal{C}=\mathcal{C}(\nu)$
for which
\begin{equation}\label{2.10+}
|V^{\nu}_n-V|\leq \mathcal{C}n^{-1/8}, \ \ \forall{n}\in\mathbb{N}.
\end{equation}
\end{thm}
Next, we describe a dynamical programming algorithm for $V^{\nu}_n$
and for the optimal control, which in general should not be unique.
Fix $n\in\mathbb{N}$ and define a sequence of functions
$J^{\nu,n}_k:(\mathbb{R}^d)^{k+1}\times(\mathbb{S}^d)^{k+1}\rightarrow
\mathbb{R},\, k=0,1,..., n$ by the backward recursion
\begin{eqnarray}\label{2.11}
&J^{\nu,n}_n(u_0,u_1,...,u_n,v_0,v_1,...,v_n)=F\left(\mathcal{W}_n(u,v)\right) \  \ \mbox{and} \\
&J^{\nu,n}_k(u_0,u_1,...,u_k,v_0,v_1,...,v_k)=\nonumber\\
&\sup_{\gamma\in\sqrt{\textbf{D}}}
\mathbb{E}\left(J^{\nu,n}_{k+1}\left(u_0,u_1,...,u_k,u_k+\frac{\gamma
Y_{k+1}}{\sqrt n}
,v_0,v_1,...,v_k,v_k+\frac{\gamma^2}{n}\right)\right)\nonumber\\
&\sup_{\gamma\in\sqrt{\textbf{D}}}\int_{\mathbb{R}^d}
J^{\nu,n}_{k+1}\left(u_0,u_1,...,u_k,u_k+\frac{\gamma x}{\sqrt
n},v_0,v_1,...,v_k,v_k
+\frac{\gamma^2}{n}\right)d\nu(x)\nonumber\\
&
\ \mbox{for} \ \ k=0,1,...,n-1.\nonumber
\end{eqnarray}
From (\ref{2.1}) and (\ref{2.6+}) it follows that there exists a
constant $\hat{H}$ such that
$$J^{\nu,n}_{k}\left(u_0,...,u_k,v_0,...,v_k\right)\leq \hat{H}
\exp\big((H_2+1)\sum_{i=0}^k (||u_i||+||v_i||)\big), \ \ \forall{k},
u_0,...,u_k,v_0,...,v_k.$$ Fix $k$. By applying (\ref{2.6+}) again
we conclude that for any compact sets $K_1\subset\mathbb{R}^d$ and
$K_2\subset\mathbb{S}^d_{+}$, the family of random variables
\begin{eqnarray*}
&J^{\nu,n}_{k+1}\left(u_0,...,u_k,u_k+\frac{\gamma Y_{k+1}}{\sqrt n}
,v_0,...,v_k,v_k+\frac{\gamma^2}{n}\right),\\
&\gamma\in\sqrt{\textbf{D}}, \ \ u_0,...,u_k\in K_1, \ \ v_0,...,v_k\in K_2\nonumber
\end{eqnarray*}
is uniformly integrable. This together with the fact that the set
$\textbf{D}$ is compact gives (by backward induction) that for any
$k$, the function $J^{\nu,n}_k$ is continuous. Thus we can introduce
the functions
$h^{\nu,n}_k:(\mathbb{R}^d)^{k+1}\times(\mathbb{S}^d)^{k+1}\rightarrow\sqrt{\textbf{D}},\,
k=0,1,..., n-1$ by
\begin{eqnarray}\label{2.12}
&h^{\nu,n}_k(u_0,...,u_k,v_0,...,v_k)=\\
&arg \max_{\gamma\in \sqrt{\textbf{D}}} \int_{\mathbb{R}^d}
J^{\nu,n}_{k+1}\left(u_0,u_1,...,u_k,u_k+\frac{\gamma x}{\sqrt n}
,v_0,v_1,...,v_k,v_k+\frac{\gamma^2}{n}\right)d\nu(x)\nonumber.
\end{eqnarray}
Finally, define by induction the stochastic processes
${\{M^{\nu,n}_k\}}_{k=0}^n$ and ${\{N^{\nu,n}_k\}}_{k=0}^n$,
with values in $\mathbb{R}^d$ and $\mathbb{S}^d$, respectively by
$M^{\nu,n}_0=0$, $N^{\nu,n}_0=0$ and for $k<n$
\begin{eqnarray}\label{2.13}
&N^{\nu,n}_{k+1}=N^{\nu,n}_{k}+\frac{1}{n}\left(h^{\nu,n}_k(M^{\nu,n}_0,...,
M^{\nu,n}_k,N^{\nu,n}_0,...,N^{\nu,n}_k)\right)^2 \\
&\mbox{and} \ \ M^{\nu,n}_{k+1}=M^{\nu,n}_{k}+\frac{1}{\sqrt
n}h^{\nu,n}_k
(M^{\nu,n}_0,...,M^{\nu,n}_k,N^{\nu,n}_0,...,N^{\nu,n}_k)Y_{k+1}.\nonumber
\end{eqnarray}
Observe that $M^{\nu,n}\in\mathcal{A}^{\nu}_n$ and $N^{\nu,n}=\langle M^{\nu,n}\rangle$.
From the dynamical programming principle it follows that
\begin{equation}\label{2.14}
V^{\nu}_n=J^{\nu,n}_0(0,0)=\mathbb{E}F\left(\mathcal{W}_n(M^{\nu,n},\langle M^{\nu,n}\rangle)\right).
\end{equation}
In the following theorem (which will be proved in Section 5) we provide an explicit construction of
$\epsilon$--optimal measures for the $G$--expectation which is defined
in (\ref{2.4}).
\begin{thm}\label{thm2.2}
Let
$(\Omega^W, \mathcal{F}^{W}, \mathbb{P}^{W})$
be a complete probability space together with a standard
$d$-dimensional Brownian motion
\{$W_t\}_{t\in[0,1]}$ and its natural filtration
$\mathcal{F}^{W}_t=\sigma{\{W(s)|s\leq{t}\}}$.
Consider the standard normal distribution $\nu_g=\mathcal{N}(0,I)$.
For any $n\in\mathbb{N}$, let $f_n:(\mathbb{R}^d)^n\rightarrow\mathbb{R}$ be a function which is satisfying
$f_n(Y^g_1,...,Y^g_n)=M^{\nu_g,n}_n$,
where $Y^g_1,...,Y^g_n$ are i.i.d. and $\mathcal{L}(Y^g_1)=\nu_g$.
Observe that $f_n$ can be calculated from (\ref{2.11})--(\ref{2.13}).
Define the continuous stochastic process ${\{M^n_t\}}_{t=0}^1$ by
\begin{equation}\label{2.15}
M^n_t=\mathbb{E}^W\left(f_n\left(\sqrt{n}W_{\frac{1}{n}},\sqrt{n}(W_{\frac{2}{n}}-W_{\frac{1}{n}}),...,\sqrt{n}(W_1-
W_{\frac{n-1}{n}})\right)|\mathcal{F}^W_t\right), \ \ t\in [0,1]
\end{equation}
where $\mathbb{E}^W$ denotes the expectation with respect to
$\mathbb{P}^{W}$. Let $P_n$ be the distribution of $M^n$ on the
canonical space $\Omega$. Then $P_n\in\mathcal{P}_D$, and for any
$\epsilon>0$ there exists a constant
$\tilde{\mathcal{C}}_{\epsilon}$ such that
\begin{equation}\label{2.16}
V<E_n F(B,\langle B\rangle)+\tilde{\mathcal{C}}_{\epsilon} n^{\epsilon-1/8}, \ \ \forall{n}
\end{equation}
where $E_n$ denotes the expectation with respect to $P_n$.
If the function $F$ is bounded then there exists a constant
 $\tilde{\mathcal{C}}$ for which
 \begin{equation}\label{2.17}
V<E_n F(B,\langle B\rangle)+\tilde{\mathcal{C}} n^{-1/8}, \ \ \forall{n}.
\end{equation}
\end{thm}

\section{The main tool}\label{sec3}\setcounter{equation}{0}
In this section we derive a strong approximation theorem (Lemma
\ref{lem3.1}) which is the main tool in the proof of Theorems
\ref{thm2.1}--\ref{thm2.2}. This theorem is an extension of the main
result in \cite{S}.

For any two distributions $\nu_1,\nu_2$ on the same measurable space
$(\mathcal{X},\mathcal{B})$ we define the distance in variation
\begin{equation}\label{3.1}
\rho(\nu_1,\nu_2)=\sup_{B\in\mathcal{B}}|\nu_1(B)-\nu_2(B)|.
\end{equation}
First we state some results (without a proof) from \cite{S} (Lemmas 4.5 and 7.2 in \cite{S}) that will be
used in the proof of Lemma \ref{lem3.1}.
\begin{lem}\label{lem3.0}
${}$\\
i. There exists a distribution $\mu$ on $\mathbb{R}^d$ which is
supported on the set $(-1/2,1/2)^d$ and has the following property.
There exists a constant $\mathcal{C}_1>0$ such that for any
distributions $\nu_1,\nu_2$ on $\mathbb{R}^d$ which satisfy
\begin{eqnarray}\label{3.2}
&\int_{\mathbb{R}^d}xd\nu_1(x)=\int_{\mathbb{R}^d}xd\nu_2(x) \ \ \mbox{and} \ \ \mbox{for} \ \ 1\leq i,j\leq d\\
&\int_{\mathbb{R}^d}x^i x^j d\nu_1(x)=\int_{\mathbb{R}^d}x^i x^j d\nu_2(x) \nonumber
\end{eqnarray}
we have
\begin{equation}\label{3.3}
\rho(\nu_1\ast \mu,\nu_2\ast\mu)\leq \mathcal{C}_1 \left(\int_{\mathbb{R}^d}||x||^3
d\nu_1(x)+\int_{\mathbb{R}^d}||x||^3d\nu_2(x)\right)
\end{equation}
where $\nu\ast\mu$ denotes the convolution of the measures $\nu$ and $\mu$.\\
ii. Let $(\tilde{\Omega},\tilde{\mathcal{F}},\tilde{P})$ be a probability space
together with a $d$--dimensional random vector $Y$, a $m$--dimensional random vector $Z$ ($m$ is some natural number),
and a random variable $\alpha$ which is distributed uniformly on the interval $[0,1]$
and independent of $Y$ and $Z$. Let $\nu$ be a distribution on $\mathbb{R}^d$ and
let $\hat\nu$ be a distribution on $\mathbb{R}^{m}\times\mathbb{R}^d$
such that $\hat\nu(A\times\mathbb{R}^d)=\tilde{P}(Z\in A)$ for
any $A\in\mathcal{B}(\mathbb{R}^m)$, i.e.
a marginal distribution of $\hat\nu$ on $\mathbb{R}^m$ is equals
to $\mathcal{L}(Z)$.
There exists a measurable
function $\Phi=\Phi_{\nu,\hat\nu,\mathcal{L}(Z,Y)}:
\mathbb{R}^m\times\mathbb{R}^d\times [0,1]\rightarrow
\mathbb{R}^d\times\mathbb{R}^d$ such that for the vector
\begin{equation}\label{3.4}
(U,X):=\Phi(Z,Y,\alpha)
\end{equation}
we have the following:
$\mathcal{L}(U)=\nu$, $\mathcal{L}(Z,X)=\hat\nu$, $U$ is independent of $X,Z$
and
\begin{equation}\label{3.6}
\tilde{P}(U+X \neq Y|Z)=\rho(\mathcal{L}(U)
\ast\mathcal{L}(X|Z),\mathcal{L}(Y|Z)).
\end{equation}
\end{lem}
Now we are ready to prove the main result of this section. For any
stochastic process $Z=\{Z_k\}_{k=0}^n$ we denote  $\Delta
Z_k:=Z_k-Z_{k-1}$ for $k\geq 0$, where we set $Z_{-1}=0$. Fix
$n\in\mathbb{N}$ and consider a $d$--dimensional martingale
${\{M_k\}}_{k=0}^n$ with respect to its natural filtration, which
satisfies $M_0=0$. For any $1\leq k\leq n$, let
$\phi_k:(\mathbb{R}^{d})^k\rightarrow\mathbb{S}^d$ be a measurable
map such that
\begin{eqnarray}\label{3.6+-}
&\sqrt{\Delta\langle M\rangle_k}=\sqrt{\mathbb{E}\left(\Delta
M_k\Delta
M_k^{'}\big|\sigma\{M_0,M_1,...,M_{k-1}\}\right)}= \\
&\phi_k(\Delta M_0,\Delta M_1,...,\Delta M_{k-1}),\nonumber
\end{eqnarray}
where $\{\langle M\rangle_k\}_{k=0}^n$ is the predictable variation
($\mathbb{S}^d_{+}$ valued) of $M$ and the symbol $\cdot^{'}$
denotes transposition. We assume that there exists a constant
 $H$ for which
\begin{equation}\label{3.6+}
\mathbb{E}\left(||\Delta
M_k||^3\big|\sigma\{M_0,...,M_{k-1}\}\right)+ ||\sqrt{\Delta\langle
M\rangle_k}||^3 \leq H, \ \ a.s. \ \ \forall{k}.
\end{equation}
\begin{lem}\label{lem3.1}
Let $\nu$ a distribution on
$\mathbb{R}^d$ such that
\begin{eqnarray}\label{3.7}
&\int_{\mathbb{R}^d}x d\nu(x)=0, \ \
\int_{\mathbb{R}^d}x^i x^j d\nu(x)=\delta_{ij} \ \ \forall{i,j\leq d}\\
&\mbox{and} \ \ \int_{\mathbb{R}^d}||x||^3 d\nu(x)<\infty.\nonumber
\end{eqnarray}
For any $\Theta>0$ its possible to construct the martingale
${\{M_k\}}_{k=0}^n$ on some probability space $(\tilde\Omega,\tilde{\mathcal{F}},\tilde{P})$
together with a sequence
of i.i.d. random vectors $Y_1,...,Y_n$ with the following properties:\\
i. $\mathcal{L}(Y_1)=\nu$.\\
ii. For any $k$, the random vectors $M_1,...,M_{k-1}$ are independent of $Y_k$.\\
iii. There exists a constant $\mathcal{C}_2=\mathcal{C}_2(\nu)$ which depends only on the distribution $\nu$ such that
\begin{equation}\label{3.9}
\tilde{P}\left(\max_{1\leq k\leq n} ||M_k-\sum_{j=1}^{k}
\sqrt{\Delta\langle M\rangle_j}Y_{j}||>\Theta\right)\leq
\frac{\mathcal{C}_2H n}{\Theta^3}.
\end{equation}
\end{lem}
\begin{proof}
Fix $\Theta>0$. For any $k$ let $\nu_k$ be the distribution of the
random vector $\frac{1}{\Theta}(\Delta M_0,...,\Delta M_k)$. Let
$(\tilde\Omega,\tilde{\mathcal{F}},P)$ be a probability space which
contains a sequence of i.i.d. random vectors $Y_1,...,Y_n$ such that
$\mathcal{L}(Y_1)=\nu$, a sequence of i.i.d. random variables
$\alpha_1,...,\alpha_n$ which are distributed uniformly on the
interval $[0,1]$ and independent of $Y_1,...,Y_n$, and a random
vector $U_0$ which is independent of
$Y_1,...,Y_n,\alpha_1,...,\alpha_n$ and satisfies
$\mathcal{L}(U_0)=\mu$, where the distribution $\mu$ is defined in
the first part of Lemma \ref{lem3.0}. Define the sequences
${\{X_i\}}_{i=0}^n$ and ${\{U_i\}}_{i=1}^n$ by the following
recursive relations, $X_0=0$ and
\begin{equation}\label{3.11}
(U_k,X_k)=\Psi_{\mu,\nu_k,\hat\nu_k}(X_0,...,X_{k-1},U_{k-1}+\frac{1}{\Theta}
\phi_k(\Theta X_0,...,\Theta X_{k-1})Y_k,\alpha_k), \ \ 1\leq k\leq n
\end{equation}
where $\hat\nu_k$ is the distribution of $(X_0,...,X_{k-1},U_{k-1}+\frac{1}{\Theta}\phi_k(\Theta X_0,...,\Theta X_{k-1})Y_k)$.
From the definition of the map $\Psi$ it follows (by induction) that $\mathcal{L}(\Theta X_0,...,\Theta X_n)=
\mathcal{L}(\Delta M_0,...,\Delta M_n)$.
We conclude that the the stochastic process ${\Theta}\sum_{i=0}^k X_i$, $0\leq k\leq n$
is distributed as ${\{M_k\}}_{k=0}^n$, and so we set,
\begin{equation}\label{3.12}
M_k={\Theta}\sum_{i=0}^k X_i, \ \ 0\leq k\leq n.
\end{equation}
Let $1\leq k\leq n$. From (\ref{3.11})--(\ref{3.12}) and the fact that $Y_k$ is independent of
$Y_1,...,Y_{k-1},$
$\alpha_1,...,\alpha_{k-1}$ it follows that
$Y_k$ is independent of $M_0,...,M_{k-1}$.
Thus in order to complete the proof, it remains to establish (\ref{3.9}).
Set
\begin{eqnarray}\label{3.13}
&\delta_k=U_k+X_k-U_{k-1}-\frac{1}{\Theta}\phi_k(\Theta X_0,...,\Theta X_{k-1})Y_k, \  \mbox{and} \
\rho_k(x_0,...,x_{k-1})\\
&=\tilde{P}(\delta_k \neq 0|X_0=x_0,...,X_{k-1}=x_{k-1}), \ \ x_0,...,x_{k-1}\in\mathbb{R}^d \ \ 1\leq k\leq n.\nonumber
\end{eqnarray}
From the properties of the map $\Psi$ it follows that for any $k$,
$U_k$ is independent of $X_0,...,X_k$ and $\mathcal{L}(U_k)=\mu$.
This together with (\ref{3.6}) and (\ref{3.11}) gives
\begin{eqnarray}\label{3.15}
&\rho_k(x_0,...,x_{k-1})=\rho\big(\mathcal{L}(X_k|X_0=x_0,...,X_{k-1}=x_{k-1})\ast \mu,\\
&\mathcal{L}(\frac{1}{\Theta}\phi(\Theta x_0,...,\Theta x_{k-1})Y_k)\ast \mu\big) \ \ x_0,...,x_{k-1}\in\mathbb{R}^d, \ \
1\leq k\leq n. \nonumber
\end{eqnarray}
From (\ref{3.3}), (\ref{3.6+-})--(\ref{3.6+}), (\ref{3.12}) and
(\ref{3.15})
\begin{equation}\label{3.16}
\rho_k(x_0,...,x_{k-1})\leq \frac{\mathcal{C}_2H}{\Theta^3},\ \ x_0,...,x_{k-1}\in\mathbb{R}^d, \ \
1\leq k\leq n
\end{equation}
for some constant $\mathcal{C}_2=\mathcal{C}_2(\nu)$ which depends only on the distribution $\nu$.
From (\ref{3.12})--(\ref{3.13}), (\ref{3.16}) and the fact that $\max_{0\leq k\leq n} ||U_k||<\frac{1}{2}$ a.s. we obtain
\begin{eqnarray*}
&\tilde{P}\left(\max_{1\leq k\leq n} ||M_k-\sum_{j=1}^{k}
\sqrt{\Delta\langle M\rangle_j}Y_{j}||>\Theta\right)=\\
&\tilde{P}\left(\max_{1\leq k\leq n} ||M_k-\sum_{j=0}^{k-1} \phi_j(\Delta M_0,...,\Delta M_j)Y_{j+1}||>\Theta\right)=\\
&\tilde{P}\left(\max_{1\leq k\leq n} \Theta||\sum_{i=1}^k\delta_i+U_0-U_k||>\Theta\right)\leq
\sum_{i=1}^n \tilde{P}(\delta_i\neq 0)\leq \frac{\mathcal{C}_2H n}{\Theta^3}
\end{eqnarray*}
and we conclude the proof.
\end{proof}

\section{Auxiliary lemmas}\label{sec4}\setcounter{equation}{0}
In this section we derive several estimates which are
essential for the proof of Theorem \ref{thm2.1}--\ref{thm2.2}.
We start with the following general result.
\begin{lem}\label{lem4.0+}
Let ${\{M_t\}}_{t=0}^1$ be a one dimensional continuous martingale which satisfies
$\frac{d\langle M\rangle_t }{dt}\leq \mathcal{H}$ a.s. for some constant $\mathcal{H}$.
Consider the discrete time martingale $N_k=M_{k/n}$, $0 \leq k\leq n$
together with its predictable variation process ${\{\langle N\rangle_k\}}_{k=0}^n$ which is given by
$\langle N\rangle _0=0$ and
$$\langle N\rangle _k=\sum_{i=1}^k \mathbb{E}((\Delta N_i)^2|\sigma\{N_0,...,N_{i-1}\}), \ \ 1\leq k\leq n.$$
There exists constants $\mathcal{C}_3,\mathcal{C}_4$ (which depend only on $\mathcal{H}$ such that)
\begin{equation}\label{4.0+}
\mathbb{E}\left(\max_{0\leq k\leq n-1}\max_{k/n\leq t\leq (k+1)/n}|M_t-N_{k}|^4\right)\leq \frac{\mathcal{C}_3}{n}
\end{equation}
and
\begin{equation}\label{4.0++}
\mathbb{E}\left(\max_{0\leq k\leq n-1}\max_{k/n\leq t\leq (k+1)/n}|\langle M\rangle _t-\langle N\rangle_k |^2\right)
\leq \frac{\mathcal{C}_4}{\sqrt n}.
\end{equation}
\end{lem}
\begin{proof}
From the Burkholder--Davis--Gundy inequality it follows that
there exists a constant $c_1$ such that
\begin{eqnarray}\label{4.4}
&\mathbb{E}\left(\max_{0\leq k\leq n-1}\max_{k/n \leq t\leq (k+1)/n} |M_t-N_k|^4\right) \leq \\
&\sum_{k=0}^{n-1} \mathbb{E}\left(\max_{ k/n \leq t \leq (k+1)/n} |M_t-M_{k/n}|^4\right)\leq \nonumber\\
&c_1\sum_{k=0}^{n-1}\mathbb{E}\left(|\langle M \rangle_{(k+1)/n}-\langle M\rangle_{k/n}|^2\right)\leq
c_1 n \frac{\mathcal{H}^2}{n^2}=\frac{c_1 \mathcal{H}^2}{n}\nonumber
\end{eqnarray}
this completes the proof of (\ref{4.0+}). Next, we prove (\ref{4.0++}).
Define the optional variation of the martingale
${\{N_k\}}_{k=0}^n$ by $[N]_0=0$ and
\begin{equation}\label{4.5}
[N]_k=\sum_{i=1}^{k}(\Delta N_i)^2, \ \ 1\leq k\leq n.
\end{equation}
From the relation $\mathbb{E}(\Delta
[N]_k|\sigma\{N_0,...,N_{k-1}\})=\Delta \langle N\rangle_k$ and the
Doob--Kolmogorov inequality we obtain
\begin{eqnarray}\label{4.6}
&\mathbb{E}\left(\max_{0\leq k\leq n}|[N]_k-\langle N\rangle _k|^2\right)\leq 4 \mathbb{E}\left(|[N]_n-\langle N\rangle_n|^2\right)=\\
&4 \mathbb{E}\left(|\sum_{i=1}^n \Delta [N]_i-\Delta \langle N\rangle_i|^2\right)=
4\sum_{i=1}^n \mathbb{E}\left(|\Delta [N]_i-\Delta\langle N\rangle_i|^2\right)\leq\nonumber\\
&4\sum_{i=1}^n \mathbb{E}((\Delta [N]_i)^2)=4\sum_{i=1}^{n}
\mathbb{E}\left(|M_{i/n}-M_{(i-1)/n}|^4\right)\leq \frac{4c_1
\mathcal{H}^2}{n}\nonumber
\end{eqnarray}
where the last inequality follows from the Burkholder--Davis--Gundy
inequality. Next, observe that
\begin{equation}\label{4.9}
[N]_k=N^2_{k}-2\sum_{i=1}^{k-1} N_{i}(N_{i+1}-N_{i})=N^2_{k}-2\int_{0}^{k/n}N_{[nt]}dM_t, \ \
1\leq k\leq n.
\end{equation}
From the Doob--Kolmogorov inequality and Ito's Isometry we get
\begin{eqnarray}\label{4.10}
&\mathbb{E}\left(\sup_{0\leq u\leq 1}\left|\int_{0}^u(M_t-N_{[nt]})dM_t\right|^2\right)\leq \\
&4\mathbb{E}\left(|\int_{0}^1 (M_t-N_{[nt]})dM_t|^2\right)=4 \mathbb{E}\left(\int_{0}^1 (M_t-N_{[nt]})^2 d\langle M\rangle_t\right)\leq\nonumber\\
&4 \mathcal{H} \mathbb{E}\left(\max_{0\leq k\leq n-1}\max_{k/n\leq
t\leq (k+1)/n}|M_t-N_{k}|^2\right)\leq \frac{4
\mathcal{H}\sqrt{\mathcal{C}_3}}{\sqrt n}, \nonumber
\end{eqnarray}
the last inequality follows from
(\ref{4.0+}) and Jensen's inequality.
From (\ref{4.9})--(\ref{4.10}) and the equality
$2\int_{0}^{k/n} M_t dM_t=N^2_k-\langle M\rangle_{k/n}$
it follows that
$$\mathbb{E}\left(\max_{1\leq k\leq n}|[N]_k-\langle M\rangle_{k/n}|^2\right)\leq \frac{16 \mathcal{H}\sqrt{\mathcal{C}_3}}{\sqrt n}.$$
This together with (\ref{4.6}) and the inequality $(a+b+c)^2\leq 4(a^2+b^2+c^2)$ yields
\begin{eqnarray}\label{4.11}
&\mathbb{E}\left(\max_{0\leq k\leq n-1}\max_{k/n\leq t\leq (k+1)/n}|\langle M\rangle _t-\langle N\rangle_k |^2\right)\leq\\
&\frac{4\mathcal{H}^2}{n^2}+4\mathbb{E}\left(\max_{1\leq k\leq n}|[N]_k-\langle M\rangle_{k/n}|^2\right)+
4\mathbb{E}\left(\max_{1\leq k\leq n}|[N]_k-\langle N\rangle _k|^2\right)\nonumber\\
&\leq \frac{4\mathcal{H}^2}{n^2}+\frac{64
\mathcal{H}\sqrt{\mathcal{C}_3}}{\sqrt n}+\frac{16c_1
\mathcal{H}^2}{n}\nonumber
\end{eqnarray}
and the proof is completed.
\end{proof}
Next, we apply the above lemma in order to derive some estimates in our setup.
\begin{lem}\label{lem4.1}
Let $n\in\mathbb{N}$ and $P\in\mathcal{P}_{\textbf{D}}$. Consider
the $d$--dimensional martingale $N_k=B_{k/n}$, $0\leq k\leq n$
together with its predictable variation ${\{\langle N
\rangle_k\}}_{k=0}^n$, under the measure $P$. There exists a
constant $\mathcal{C}_5$ (which is independent of $n$ and $P$) such
that
\begin{equation}\label{4.12}
{E}_{P}\left(||\mathcal{W}_n(N)-B||^2\right)\leq \frac{\mathcal{C}_5}{\sqrt n}
\end{equation}
and
\begin{equation}\label{4.13}
{E}_{P}\left(||\mathcal{W}_n(\langle N\rangle)-\langle B\rangle||^2\right)\leq \frac{\mathcal{C}_5}{\sqrt n}.
\end{equation}
In the equations (\ref{4.12}) and (\ref{4.13}), $\mathcal{W}_n$
is the linear interpolation operator which is defined on the spaces $(\mathbb{R}^d)^{n+1}$ and
$(\mathbb{S}^d)^{n+1}$, respectively.
\end{lem}
\begin{proof}
Inequality (\ref{4.12}) follows immediately from (\ref{4.0+}) and the relation
$$||\mathcal{W}_n(N)-B||\leq 2 \sum_{i=1}^d\max_{1\leq k\leq n}\max_{k/n\leq t\leq (k+1)/n}|N^i_k-B^i_{t}|.$$
Next, we prove (\ref{4.13}).
For any $1\leq i,j\leq d$ denote by
$\langle N\rangle^{i,j}_k$ and $\langle B\rangle^{i,j}_t$,
the $i-$th row and the $j-$th column of the
matrices $\langle N\rangle _k$ and $\langle B\rangle_t$, respectively.
Notice that
$\langle B\rangle^{i,j}_t=\frac{1}{2}(\langle B^i+B^j\rangle_t-\langle B^i\rangle_t-\langle B^j\rangle_t)$
and $\langle N\rangle^{i,j}_k=\frac{1}{2}(\langle N^i+N^j\rangle_k-\langle N^i\rangle_k-\langle N^j\rangle_k)$.
Thus (\ref{4.13}) follows from (\ref{4.0++}) and the inequality
$$||\mathcal{W}_n(\langle N\rangle)-\langle B\rangle||\leq
2\sum_{i=1}^d\sum_{j=1}^d\max_{0\leq k\leq n-1}\max_{k/n\leq t\leq
(k+1)/n}|\langle N\rangle^{i,j}_k-\langle B\rangle^{i,j}_t|.$$
\end{proof}
We conclude this section with the following technical lemma.
\begin{lem}\label{lem4.2}
Let $A>0$. Then we have:\\
i.
\begin{equation}\label{4.14}
\sup_{P\in\mathcal{P}_{\textbf{D}}}E_P\exp(A\sup_{0\leq t\leq 1}||B_t||)<\infty.
\end{equation}
ii. Let $n\in\mathbb{N}$ and $\nu$ be a distribution which satisfies (\ref{2.6})--(\ref{2.6+}).
Consider a filtered probability space $(\tilde\Omega,\tilde{\mathcal{F}},\{\tilde{\mathcal{F}}_k\}_{k=0}^n,\tilde{P})$
together with a sequence of i.i.d. random vectors $Y_1,...,Y_n$ which satisfy
$\mathcal{L}(Y_1)=\nu$. Assume that for any $i$, $Y_i$ is $\tilde{\mathcal{F}}_i$
measurable and independent of $\tilde{\mathcal{F}}_{i-1}$. Let ${\{M_k\}}_{k=0}^n$
be a $d$--dimensional stochastic process of the following form:
$M_0=0$ and
\begin{equation}\label{4.14+}
M_k=\sqrt\frac{1}{n}\sum_{i=1}^k\gamma_i Y_i, \ \ 1\leq k\leq n
\end{equation}
where for any $i$, $\gamma_i$ is $\tilde{\mathcal{F}}_{i-1}$
measurable random matrix, which takes values
in $\sqrt{\textbf{D}}$.
There exists a constant $\mathcal{C}_6$ (which may depend on $A$ and $\nu$)
such that
\begin{equation}\label{4.15}
\exp\left(A \max_{0\leq k\leq n}||M_k||\right)<\mathcal{C}_6.
\end{equation}
\end{lem}
\begin{proof}
i. Let $P\in\mathcal{P}_{\textbf{D}}$.
From the Novikov condition it follows that for any $1\leq i\leq d$ and $a\in\mathbb{R}$,
$E_P\exp\left(aB^i_1-\frac{a^2}{2}\langle B^i\rangle_1\right)=1$. Thus
$$E_P\left(\exp(a|B^i_1|)\right)\leq E_P(\exp(aB^i_1))+E_P(\exp(-aB^i_1))\leq 2
\exp\left(\frac{a^2}{2}||\textbf{D}||\right)$$
where $||\textbf{D}||=\sup_{\mathcal{D}\in\textbf{D}}||\mathcal{D}||$.
This together with the Cauchy--Schwartz
inequality completes the proof of (\ref{4.14}).\\
ii. Consider the compact set $K:=\{x\in\mathbb{R}^d: ||x||\leq
||\sqrt \textbf{D}||\}$. Clearly, the rows of the matrices
$\gamma_j$, $1\leq j\leq n$ are in $K$. Fix $1\leq i\leq d$ and
consider the $i-$th component of the process $M$, namely we consider
the process $(M^i_0,...,M^i_n)$. From (\ref{4.14+}) we get that for
any $a\in\mathbb{R}$
$$\mathbb{E}\left(\exp(a(M^i_k-M^i_{k-1}))|\tilde{\mathcal{F}}_{k-1}\right)\leq \sup_{y\in K}\psi_{\nu}\left(\frac{ay}{\sqrt n}\right)$$
where $\psi_{\nu}$ is the function which is defined below (\ref{2.6}).
This together with (\ref{2.6+}) gives
\begin{equation}\label{4.17}
\mathbb{E}\left(\exp(aM^i_n)\right)\leq \sup_{n\in\mathbb{N}}\sup_{y\in K}\psi^n_{\nu}\left(\frac{ay}{\sqrt n}\right)<\infty.
\end{equation}
From the inequality $\mathbb{E}\exp(|a M^i_n|)\leq
\mathbb{E}\left(\exp(aM^i_n)\right)+\mathbb{E}\left(\exp(-aM^i_n)\right)$
and the Cauchy--Schwartz inequality it follows that there exists a
constant $c_2$ (which may depend on $A$ and $\nu$) such that
\begin{equation}\label{4.17+-}
\mathbb{E}(\exp(A||M_n||))< c_2.
\end{equation}
Finally, since for any $i$ the process $M^i_k$, $k\leq n$ is a
martingale with respect to the filtration
$\{\tilde{\mathcal{F}}_k\}_{k=0}^n$ we conclude that the stochastic
process ${\{\exp(A||M_k||/2)\}}_{k=0}^n$ is a sub--martingale and
so, from (\ref{4.17+-}) and the Doob--Kolomogorv inequality
$\mathbb{E}\exp(A \max_{0\leq k\leq n}||M_k||)\leq 4 c_2$ and the
proof is completed.
\end{proof}

\section{Proof of the main results}\label{sec5}\setcounter{equation}{0}
In this section we complete the proof of Theorems \ref{thm2.1}--\ref{thm2.2}.
Let $\nu$ be a distribution which satisfies (\ref{2.6})--(\ref{2.6+}).
Fix $\epsilon>0$.
We start with proving the following statements
\begin{equation}\label{5.1}
V^{\nu}_n>V-\mathcal{C}_{\epsilon} n^{\epsilon-1/8}, \ \ \forall{n}\in\mathbb{N}
\end{equation}
and for a bounded $F$
\begin{equation}\label{5.1+}
V^{\nu}_n>V-\mathcal{C}n^{-1/8}, \ \ \forall{n}\in\mathbb{N}.
\end{equation}
Choose $n\in\mathbb{N}$ and $\delta>0$.
There exists a measure $Q\in \mathcal{P}_{\textbf{D}}$
for which
\begin{equation}\label{5.2}
V<\delta+E_QF(B,\langle B\rangle).
\end{equation}
Consider the stochastic process $N_k=B_{k/n}$, $0\leq k\leq n$
together with its predictable variation ${\{\langle
N\rangle_k\}}_{k=0}^n$. From the fact that $\textbf{D}$ is a convex
compact set we obtain that there exists a sequence of functions
$\psi_j:(\mathbb{R}^{d})^{j} \rightarrow\sqrt{\textbf{D}}$, $1\leq
j\leq n$ such that
\begin{eqnarray}\label{5.3-}
&\sqrt{\Delta\langle N\rangle_k}=\sqrt{\mathbb{E}\left(\Delta
N_k\Delta N_k^{'}\big|\sigma\{N_0,N_1,...,N_{k-1}\}\right)}=\\
& \frac{1}{\sqrt n}\psi_k(N_0,...,N_{k-1}), \ \ \forall{k} \ \
\mbox{a.s.}\nonumber
\end{eqnarray}
From the Burkholder--Davis--Gundy inequality it follows that there
exists a constant $c_3$ for which
\begin{equation}\label{5.3+}
{E}_Q\left(||\Delta N_k||^3\big|\sigma\{N_0,...,N_{k-1}\}\right)\leq
c_3n^{-3/2}, \ \  \forall{k} \ \ \mbox{a.s.}
\end{equation}
By applying (\ref{2.1}), Lemmas \ref{lem4.1}--\ref{lem4.2}
and Cauchy--Schwartz inequality we get
\begin{equation}\label{5.4}
E_Q|F(B,\langle B\rangle)-F(\mathcal{W}_n(N),\mathcal{W}_n(\langle
N\rangle))| \leq c_4 n^{-1/4}
\end{equation}
for some constant $c_4$ (which depends only on the distribution
$\nu$). From (\ref{5.3+}) and Lemma \ref{lem3.1} we obtain that
there exists a probability space
$(\tilde\Omega,\tilde{\mathcal{F}},\tilde{P})$ which contains the
martingale $N$, a sequence of i.i.d. random vectors $Y_1,...,Y_n$
and satisfies, $\mathcal{L}(Y_1)=\nu$, for any $k$ the random
vectors $N_1,...,N_{k-1}$ are independent of $Y_k$, and
\begin{equation}\label{5.5}
\tilde{P}\left(\max_{1\leq k\leq n}
||N_k-\sum_{j=1}^{k}\sqrt{\Delta\langle
N\rangle_j}Y_j>n^{-1/8}\right)<\frac{c_5n^{-3/2}n}  {n^{-3/8}}=c_5
n^{-1/8}
\end{equation}
for some constant $c_5$ which depends only on the distribution
$\nu$. Denote $M_k=\sum_{j=1}^{k}\sqrt{\Delta\langle
N\rangle_j}Y_j$, $1\leq k\leq n$ and $\mathcal{A}=\{\max_{1\leq
k\leq n} ||N_k-M_k||>n^{-1/8}\}.$ From (\ref{2.6}) and the fact that
$N_1,...,N_{k-1}$ are independent of $Y_k$ we obtain that $M$ is a
martingale, and $\langle M\rangle=\langle N\rangle$. Thus from
(\ref{2.1}), (\ref{5.5}), Lemma \ref{lem4.2}, the Markov inequality
and the Holder inequality (for $p=\frac{1}{1-8\epsilon}$ and
$q=\frac{1}{8\epsilon}$) we get that there exists constants
$c_6,c_7$ which depend on $\epsilon$ and $\nu$ such that
\begin{eqnarray}\label{5.5+}
&\tilde{E}|F(\mathcal{W}_n(N),\mathcal{W}_n(\langle N\rangle))-F(\mathcal{W}_n(M),\mathcal{W}_n(\langle M\rangle))|\leq  \\
&H_1 \tilde{E}\bigg(\exp\left(H_2(\max_{1\leq k\leq n}||M||_k+\max_{1\leq k\leq n}||N||_k+2||\textbf{D}||)\right)\times\nonumber\\
&\left(n^{-1/8}+ \mathbb{I}_{\mathcal{A}}(||\mathcal{W}_n(N)||+||\mathcal{W}_n(M)||)\right)\bigg) \leq\nonumber\\
&\leq c_6
(n^{-1/8}+\tilde{P}\mathcal(A)^{\frac{1}{1-8\epsilon}})\leq c_7
n^{\epsilon-1/8}\nonumber
\end{eqnarray}
where we set $\mathbb{I}_{\mathcal{A}}=1$ if an event $\mathcal{A}$
occurs and $\mathbb{I}_{\mathcal{A}}=0$ if not, and $\tilde{E}$
denotes the expectation with respect to $\tilde{P}$. If the function
$F$ is bounded, say $F\leq R$, then we have
\begin{eqnarray}\label{5.7}
&\tilde{E}|F(\mathcal{W}_n(N),\mathcal{W}_n(\langle N\rangle))-F(\mathcal{W}_n(M),\mathcal{W}_n(\langle M\rangle))|\leq R \tilde{P}(A)+H_1 n^{-1/8}  \\
& \times\tilde{E}\left(\exp(H_2(\max_{1\leq k\leq
n}||M||_k+\max_{1\leq k\leq n}||N||_k+2||\textbf{D}||))\right)\leq
c_8 n^{-1/8}\nonumber
\end{eqnarray}
for some constant $c_8$ which depends only on $\nu$. Since
$\delta>0$ was arbitrary, then in view of (\ref{5.2}), (\ref{5.4})
and (\ref{5.5+})--(\ref{5.7}) we conclude that in order to prove
(\ref{5.1})--(\ref{5.1+}) it remains to establish the following
inequality
\begin{equation}\label{5.8}
V^{\nu}_n\geq\tilde{E}F(\mathcal{W}_n(M),\mathcal{W}_n(\langle M\rangle)).
\end{equation}
Define a sequence of functions
$L_k:(\mathbb{R}^d)^{k+1}\times(\mathbb{S}^d)^{k+1}\rightarrow \mathbb{R},\, k=0,1,...,
n$ by the backward recursion
\begin{eqnarray}\label{5.9}
&L_n(u_0,...,u_n,v_0,...,v_n)=F\left(\mathcal{W}_n(u,v)\right) \  \ \mbox{and} \\
&L_k(u_0,...,u_k,v_0,...,v_k)=\nonumber\\
& \tilde{E} L_{k+1}\big(u_0,...,u_k,u_k+\frac{1}{\sqrt
n}\psi_{k+1}(u_0,...,u_k)
Y_{k+1},v_0,...,v_k,\nonumber\\
&v_k+\frac{1}{n}\psi^2_{k+1}(u_0,...,u_k)\big)\nonumber \ \
\mbox{for} \ \ k=0,1,...,n-1.\nonumber
\end{eqnarray}
From the fact that $Y_{k+1}$ is independent
of $Y_1,...,Y_k,N_1,...,N_{k-1}$ it follows (by backward induction) that for any $k$,
\begin{eqnarray}\label{5.10}
&\tilde{E}\left(F(\mathcal{W}_n(M),\mathcal{W}_n(\langle M\rangle))|\sigma\{N_1,...,N_{k-1},Y_1,...,Y_{k}\}\right)=\\
&L_k\left(M_0,...,M_k,\langle N\rangle_0,...,\langle N\rangle _k\right).\nonumber
\end{eqnarray}
Finally, from (\ref{2.11}), (\ref{5.9})--(\ref{5.10}) and the fact
that $\psi_k$ takes values in $\sqrt{\textbf{D}}$ for any $k$, we
obtain (by backward induction) that $L_k\leq J^{\nu,n}_k$, $k\leq
n$, and in particular
\begin{equation}\label{5.10+}
V^{\nu}_n=J^{\nu,n}_0(0,0)\geq L_0(0,0)=\tilde{E}F(\mathcal{W}_n(M),\mathcal{W}_n(\langle M\rangle)).
\end{equation}
This competes the proof of (\ref{5.1})--(\ref{5.1+}). Next, fix
$n\in\mathbb{N}$, a distribution $\nu$ which satisfies
(\ref{2.6})--(\ref{2.6+}) and consider the optimal control
$M^{\nu,n}$ which is given by (\ref{2.11})--(\ref{2.13}). By
applying Lemma \ref{lem3.1} for the standard normal distribution
$\nu_g$ it follows that there exists a probability space
$(\tilde\Omega,\tilde{\mathcal{F}},\tilde{P})$ which contains the
martingale $M^{\nu,n}$, a sequence of i.i.d. standard Gaussian
random vectors ($d$--dimensional) $Y^g_1,...,Y^g_n$ such that for
any $k$ the random vectors $M^{\nu,n}_1,...,M^{\nu,n}_{k-1}$ are
independent of $Y^g_k$, and
\begin{equation}\label{5.10+-}
\tilde{P}\left(\max_{1\leq k\leq n}
||M^{\nu,n}_k-\sum_{j=1}^k\sqrt{\Delta \langle M^{\nu,n}\rangle_j}
Y^g_j ||>n^{-1/8}\right)<c_{9} n^{-1/8}
\end{equation}
for some constant $c_{9}$. Denote
$\hat{M}_k=\sum_{j=1}^k\sqrt{\Delta \langle M^{\nu,n}\rangle_j}
Y^g_j$, $1\leq k\leq n$. Observe that $\langle \hat M\rangle=\langle
M^{\nu,n}\rangle$. Thus by using similar argument to those as in
(\ref{5.5+})--(\ref{5.7}) we obtain that there exists constants
$c_{10},c_{11}$ such that
\begin{eqnarray}\label{5.10++}
&|\tilde{E}F(\mathcal{W}_n(\hat M),\mathcal{W}_n(\langle \hat M\rangle))-V^{\nu}_n|\leq\\
&\tilde{E}|F(\mathcal{W}_n(\hat M),\mathcal{W}_n(\langle \hat
M\rangle))-F(\mathcal{W}_n(M^{\nu,n}),\mathcal{W}_n(\langle
M^{\nu,n}\rangle))|\leq c_{10} n^{\epsilon-1/8}\nonumber
\end{eqnarray}
and
if the function $F$ is bounded,
\begin{eqnarray}\label{5.10+++}
&|\tilde{E}F(\mathcal{W}_n(\hat M),\mathcal{W}_n(\langle \hat M\rangle))-V^{\nu}_n|\leq\\
&\tilde{E}|F(\mathcal{W}_n(\hat M),\mathcal{W}_n(\langle \hat M\rangle))-F(\mathcal{W}_n(M^{\nu,n}),\mathcal{W}_n(\langle M^{\nu,n}\rangle))|\leq
 c_{11} n^{-1/8}.\nonumber
\end{eqnarray}
By applying similar arguments to those as in (\ref{5.9})--(\ref{5.10+}) we conclude that
\begin{equation}\label{5.10++++}
V^{\nu_g}_n=J^{\nu_g,n}_0(0,0)\geq \tilde{E}F(\mathcal{W}_n(\hat M),\mathcal{W}_n(\langle \hat M\rangle)).
\end{equation}
Next, let
 $z_k:(\mathbb{R}^d)^k\rightarrow\sqrt{\textbf{D}}$,
$1\leq k\leq {n-1}$ be a sequence of functions such that for any
$1\leq k\leq n-1$, $z_k(Y^g_1,...,Y^g_k)=h^{\nu_g,n}_k
(M^{\nu_g,n}_0,...,M^{\nu_g,n}_k,N^{\nu_g,n}_0,...,N^{\nu_g,n}_k),$
where the terms $M^{\nu_g,n},N^{\nu_g,n}$ are given by
(\ref{2.11})--(\ref{2.13}). From the martingale representation
theorem if follows that the martingale $M_n$ which is defined by
(\ref{2.15}) equals to
\begin{eqnarray*}
&M^n_t=h^{\nu_g,n}_0(0,0)W_{t}+\mathbb{I}_{t>1/n}\times\\
&\int_{1/n}^t z_{[nu]}(\sqrt{n}W_{1/n},\sqrt{n}(W_{2/n}-W_{1/n}),...,\sqrt{n}(W_{[nu]}-W_{[nu]-1}))dW_u
, \ \ t\in [0,1] \nonumber
\end{eqnarray*}
and so we obtain that $P_n\in\mathcal{P}_{\textbf{D}}$.
As in (\ref{5.4})
we have
\begin{equation}\label{5.11}
E_n |F(B,\langle B\rangle)- F(\mathcal{W}_n(N),\mathcal{W}_n(\langle
N \rangle))|\leq c_4 n^{-1/4}
\end{equation}
where $N_k=B_{k/n}$, $0\leq k\leq n$.
Finally, observe that the distribution of $N$ under $P_n$
equals to the distribution of
the martingale $M^{\nu_g,n}$. Thus from (\ref{2.14}) and (\ref{5.11}) we conclude that
$$V\geq E_{P_n}F(B,\langle B\rangle)\geq V^{\nu_g}_n-c_4 n^{-1/4}.$$
This together with (\ref{5.1})--(\ref{5.1+}) and (\ref{5.10++})--(\ref{5.10++++})
completes the proof of Theorems \ref{thm2.1}--\ref{thm2.2}.
\qed

\end{document}